\documentclass{amsart}
\usepackage{eufrak}

\usepackage{amssymb,amsmath, amscd, mathrsfs, yfonts, bbm,
graphics}

\usepackage{dsfont}

\newtheorem{claim}{}[section]

\newtheorem{defn}[claim]{Definition}

\newtheorem{thm}[claim]{Theorem}

\newtheorem{remark}[claim]{Remark}
\newtheorem{prop}[claim]{Proposition}

\newtheorem{cor}[claim]{Corollary}

\newcommand{\B}{\mathrm{Ber}}

\newcommand{\lra}{\longrightarrow}
\newcommand{\cB}{\mathcal{B}}

\newcommand{\X}{\mathcal{X}}
\newcommand{\ga}{\mathfrak{a}}
\newcommand{\gs}{\mathfrak{s}}
\newcommand{\cA}{\mathcal{A}}
\newcommand{\bx}{\cB\langle \X\rangle}
\newcommand{\ba}{\textbf{a}}
\newcommand{\gH}{\mathfrak{H}}
\newcommand{\mG}{\mathcal{G}}

\newcommand{\Nilp}{\text{Nilp}}

\newcommand{\ncspace}[1]{\ensuremath{#1}_{\text{nc}}}

\newcommand{\Sb}{\Sigma_{\cB}}

\title[Operator-valued semicircle, arcsine and
Bernoulli laws]{On the operator-valued analogues of the semicircle, arcsine and
Bernoulli laws}

\author{S. T. Belinschi}\address{Department of Mathematics and Statistics, University of Saskatchewan,
106 Wiggins Road, Saskatoon, SK, S7N 5E6, CANADA, and
\newline
Institute of Mathematics ``Simion Stoilow'' of the Romanian Academy.
}
\email{belinschi@math.usask.ca}
\author{M. Popa}
\address{Center for Advanced Studies in Mathematics at the Ben Gurion  University of Negev, P.O. B. 653, Be'er Sheva 84105, Israel,
and
\newline Institute of Mathematics ``Simion Stoilow'' of the Romanian Academy, P.O. Box 1-764, Bucharest, RO-70700, Romania}
\email{popa@math.bgu.ac.il}
\author{V. Vinnikov}
\address{Department of Mathematics, Ben Gurion University of Negev, Be'er Sheva 84105, Israel}
\email{vinnikov@cs.bgu.ac.il}
\thanks{Research of STB was supported by a Discovery grant from the Natural Sciences and Engineering Research Council of Canada, and a University of Saskatchewan start-up grant.}

\begin{document}

\begin{abstract}
We study of the connection between operator valued central limits for monotone, Boolean and free probability theory, which we shall call the arcsine, Bernoulli and semicircle distributions, respectively. In scalar-valued non-commutative probability these measures are known to satisfy certain arithmetic relations with respect to Boolean and free convolutions. We show that generally the corresponding operator-valued distributions satisfy the same relations only when we consider them in the fully matricial sense introduced by Voiculescu. In addition, we provide a combinatorial description in terms of moments of the operator valued arcsine distribution and we show that its reciprocal Cauchy transform satisfies a version of the Abel equation similar to the one satisfied in the scalar-valued case.
\end{abstract}

\maketitle

\section{Introduction}

 By a non-commutative probability space we shall understand a pair $(\cA, \varphi)$ where $\cA$ is a unital $\ast$-algebra over the complex numbers and $\varphi:\cA\lra\mathbb{C}$ is a positive functional with $\varphi(1)=1$.
 If $\cB$ is a unital C$^\ast$-algebra, a $\cB$-valued non-commutative probability space is a double $(\cA, E_\cB)$, where $\cA$ is a $\ast$-algebra containing $\cB$ as a $\ast$-subalgebra and $E_\cB$ is a positive conditional expectation from $\cA$ onto $\cB$. If $\cB\subset\cA$ is an inclusion of unital C$^\ast$-algebras, then $(\cA, E_\cB)$ will be called a $\cB$-valued non-commutative C$^\ast$-probability space.  Elements $X\in\mathcal A$ are called {\emph{random variables} or (in the second context) {\emph{ $\mathcal B$-valued} (or \emph{operator-valued}) \emph{ random variables}.

 We will denote by $\bx$ the $\ast$-algebra freely generated by $\cB$ and the self-adjoint symbol $\X$ (that is the $\ast$-algebra of non-commutative polynomials in $\X$ and with coefficients in $\cB$).  We will also use the notations $\cB\langle  \X \rangle _0$ for the $\ast$-subalgebra of $\bx$ of all polynomials without a free term, and the notation $\cB\langle \X_1, \X_2, \dots\rangle$ for the $\ast$-algebra freely generated by $\cB$ and the non-commutating self-adjoint symbols $\X_1, \X_2, \dots$.   If $X\in\cA$ is a self-adjoint element, then we will also use the notations $\cB\langle X\rangle$ and $\cB\langle X\rangle_0$ for the $\ast$-algebra generated as above by $X$ and $\cB$. The set of all positive conditional expectations from $\bx$ to $\cB$ will be denoted by $\Sb$. The \emph{fully matricial extension} of $\mu\in\Sb$ is the sequence
$\widetilde{\mu}=\{\mu\otimes1_n\}_n$, where $1_n$ stands for the identity in $M_n(\mathbb{C})$. Note that $\mu$ is uniquely determined by the symmetric moments of its fully matricial extension, since (see \cite{ncfound}, \cite{V1}) for

    \[
 b=\left[
 \begin{array}{ccccc}
 0&b_1&0&\dots&0\\
 0&0&b_2&\dots&0\\
 \dots&\dots&\dots&\dots&\dots\\
 0&0&0&\dots&b_n\\
 0&0&0&\dots&0
\end{array}
 \right]
 \in M_{n+1}(\cB)
 \]
 we have that
 \[
  \widetilde{\mu}\left([\X\cdot b]^n\right)=\left[
\begin{array}{cccc}
0&\dots&0&\mu(\X b_1 \X b_2\dots \X b_n)\\
0&\dots&0&0\\
\dots&\dots&\dots&\dots\\
0&\dots&0&0
\end{array}
\right].
  \]
 where $\bx$ acts at left and right on $\bx\otimes M_{m}(\mathbb{C}$) by entrywise multiplication.

  We will also denote  $\Sb^0$ the set of all $\mu\in\Sb$  whose moments do not grow faster than exponentially, that is there exists some $M>0$ such that, for all  positive integers $m$, all $b_1, \dots, b_n\in M_m(\cB)$ we have that
 \[
 ||\widetilde{\mu}(\X b_1 \X b_2\cdots \X b_n \X)||< M^{n+1}||b_1||\cdots ||b_n||.
  \]

If $(\cA, \varphi)$ is a non-commutative C${}^\ast$-probability space, the distribution of a self-adjoint element (or non-commutative random variable) $X$ of $\cA$
is a real measure $\mu_X$ described via
  \[ \int t^n d\mu_X(t)=\varphi(X^n).\]
In the more general case of a $\cB$-valued non-commutative probability space, as shown in \cite{V*}, the appropriate analogue for the distribution of a self-adjoint $X\in\cA$ is $\mu_X\in\Sb$, given by
 \[
 \mu_X(f(\X))=E_\cB(f(X))\ \text{for all}\ f(\X)\in\bx.
 \]
If $X$ is an elements of a $\cB$-valued non-commutative C$^\ast$-probability space, then $\mu_X\in\Sb^0$; moreover, for each $\mu\in\Sb^0$, there exist some element $X$ in a $\cB$-valued C$^\ast$-non-commutative probability space such that $\mu_X=\mu$ (see \cite[Proposition 1.2]{popavinnikov} and \cite{V*}).

This material presents some properties of three remarkable elements from $\Sb^0$, namely the operator-valued semicircular, Bernoulli and arcsine distributions, which are the Central Limit Laws for Free, Boolean and Monotone independence. We shall denote them by $\mathfrak s,\B$ and $\mathfrak a$, respectively. It is known that in scalar-valued non-commutative probability, we have $\mathfrak a=\mathfrak s^{\uplus 2}=\B^{\boxplus 2}$, where $\uplus$ and $\boxplus$ denote Boolean and free additive convolutions, respectively. We show that these relations extend to the operator valued context only when properly understood in the fully matricial set-up introduced by Voiculescu \cite{V2,V1} (Proposition \ref{prop8}). In addition, we provide a new combinatorial description (in terms of moments) of $\mathfrak a$ (Theorem \ref{ga}). Moreover, as monotone convolution of operator-valued distributions is shown \cite{P} to be described in terms of the composition of reciprocals of operator-valued Cauchy transforms, it is natural to inquire whether a linearization of this composition of functions similar to the Abel equation described in \cite{PommerenkeBaker,Pommerenke} holds in the operator-valued case. The positive answer is provided in Theorem \ref{4}.

The rest of the introduction is dedicated to defining the main notions and tools to be used in the paper. In the second section we give brief descriptions of the operator-valued Bernoulli and semicircular distributions and some of their transforms, as well as a new characterization of the moments of the arcsine distribution. Finally, in the third section we discuss the connection between the three central limits and show that the reciprocal of the Cauchy transform of the arcsine distribution satisfies a version of the Abel equation.

\subsection{Independence relations and transforms}
 Since the paper deals with relevant elements from $\Sb^0$, we will present the Free, Boolean and Monotone independences in a C$^\ast$-algebraic context.

 \begin{defn}
 Let $(\cA, E_\cB)$ be a $\cB$-valued non-commutative C$^\ast$-probability space and $\{X_i\}_{i\in I}$ be a family of self-adjoint elements from $\cA$.
 \begin{enumerate}
 \item[(a)] the family $\{X_i\}_{i\in I}$ is said to be \emph{free independent} over $\cB$ if
 \[ E_\cB(A_1\cdots A_n)=0\]
 whenever $A_j\in\cB\langle X_{\epsilon(j)}\rangle\cap\text{Ker}(E_\cB)$, $\epsilon(j)\in I, \epsilon(k)\neq\epsilon(k+1)$.
 \item[(b)] the family $\{X_i\}_{i\in I}$ is said to be \emph{Boolean independent} over $\cB$ if
 \[ E_\cB(A_1\cdots A_n)=E_\cB(A_1)\cdots E_\cB(A_n)\]
 whenever $A_j\in\cB\langle X_{\epsilon(j)}\rangle_0$, $\epsilon(j)\in I, \epsilon(k)\neq\epsilon(k+1)$.
 \item[(c)] if the set of indices $I$ is totally ordered, then the family
 $\{X_i\}_{i\in I}$ is said to be \emph{monotone independent} over $\cB$ if
 \[ E_\cB(A_1\cdots A_{j-1}\cdot A_j\cdot A_{j+1}\cdots A_n)= E_\cB(A_1\cdots A_{j-1}\cdot E_\cB(A_j)\cdot A_{j+1}\cdots A_n)\]
 whenever $A_l\in\cB\langle X_{\epsilon(l)}\rangle_0$, $\epsilon(l)\in I, \epsilon(l)\neq\epsilon(l+1)$ and
 $\epsilon(j-1)<\epsilon(j)>\epsilon(j+1)$, $1\leq j\leq n$.
 \end{enumerate}
  If $X, Y$ are two free (Boolean, respectively monotone independent) $\cB$-valued non-commutative random variables, then $\mu_{X+Y}$ depends only on $\mu_X$ and $\mu_Y$ and is said to be the additive free convolution $\mu_X\boxplus\mu_Y$ (additive Boolean convolution $\mu_X\uplus\mu_Y$, respectively additive monotone convolution $\mu_X\triangleright\mu_Y$) of $\mu_X$ and $\mu_Y$.

 \end{defn}

  Note that $E_{M_n(\cB)}=E_\cB\otimes 1_n:M_n(\cA)\lra M_n(\cB)$ is still a positive conditional expectation for all positive integers $n$ and any linear functional (in particular any trace $\tau$) on $\cB$ extends to $\tau\otimes{\rm tr}_n\colon\mathcal B\otimes M_n(\mathbb C)\to\mathbb C$, where ${\rm tr}_n$ is the canonical normalized trace on $M_n(\mathbb C)$. Note also that if $X, Y\in\mathcal A$ are free, Boolean, respectively monotone independent with respect to $E_\cB$, then so are $X\otimes 1_n$ and $Y\otimes 1_{n}$ with respect to $E_{M_n(\mathcal B)}$.

  We will denote $\Nilp(\cB)=\coprod_{n=1}^\infty Nilp({\cB, n})$, where $\Nilp(\cB,n)$ is the set of all $T\in M_n(\cB)$ such that $T^r=0$ for some $r$, where we view $T$ as a matrix over the \emph{tensor algebra} over $\cB$ (see \cite{popavinnikov}, \cite{BPV1}). For a given $\mu\in\Sb^0$, we define its \emph{moment-generating series} as the function $M_\mu$ given, for $b\in M_n(\cB)$, by
  \[ M_\mu(b)=\sum_{k=0}^\infty (\widetilde{\mu}\left([\X\cdot b]^k\right)
  =
1_n+(\mu\otimes1_n)(\X\cdot  b)
+(\mu\otimes1_n)(\X\cdot b\cdot\X\cdot b)+\cdots.
\]
 We define the $R$-, $B$-, respectively \emph{$\gH$-transforms} of $\mu$ via the functional equations
 \begin{eqnarray}
 M_\mu(b)-\mathds{1}&=&R_\mu\left(b\cdot M_\mu(b))\right)\label{eqmr}\\
 M_\mu(b)-\mathds{1}&=&B_\mu(b)\cdot M_\mu(b)\label{eqmb}\\
 \gH_\mu(b)&=&b\cdot M_\mu(b)\nonumber
 \end{eqnarray}
 where the notation $\mathds{1}$ stands for $1_n$ on each component from $M_n(\cB)$.

   As shown in \cite{popavinnikov},  each $M_\mu, R_\mu, B_\mu$, $\gH_\mu$ is well defined on $\Nilp(\cB)$ and on a correspondent small open ball around the origine from $\ncspace{\cB}$ which is mapped in another open ball around the origine from $\ncspace{\cB}$.

   If $X$ is a selfadjoint element from a $\cB$-valued non-commutative C$^\ast$-probability space $(\cA, E_\cB)$, we will denote $M_X$, $R_X$ etc for the correspondent transforms of $\mu_X$. The main reason for which we have introduced the $R$ and $B$-transforms is their linearizing property. Namely, if $X, Y$ are free independent over $\cB$, then (see \cite{V*,V1})
   \[R_{X+Y}(b)=R_X(b)+R_Y(b)\]
   and if $X, Y$ are Boolean independent over $\cB$, then (see \cite{mvbool})
   \[B_{X+Y}(b)=B_X(b)+B_Y(b).\]
  Moreover, if $X,Y$ are monotone independent over $\cB$ (in this order), we have that (see \cite{P})
  \[\gH_{X+Y}(b)=\gH_X\circ\gH_Y(b).\]

 Another object that will be used in the following sections is \emph{the generalized resolvent,} or \emph{operator-valued Cauchy transform}(see \cite{V2} and for new applications, \cite{BPV1}), namely the map
 \begin{align*}
\mG_\mu\colon& \mathbb{H}^{+}(\ncspace{\cB})\lra \mathbb{H}^{-}(\ncspace{\cB})\\
\mG_\mu(b)&=\widetilde{\mu}\left([b-\X\cdot\mathds{1}]^{-1}\right)\ \text{if}\ b\in \mathbb{H}(M_n(\cB))
\end{align*}
where if $\mathcal{C}$ is a $\ast$-algebra, then $\mathbb{H}^{+}(\mathcal{C})=\{a\in\mathcal C, \Im a=(a-a^\ast)/2i>0\} $ and $\mathbb{H}^{+}(\ncspace{\cB})=\coprod_{n=1}^\infty \mathbb{H}^{+}(M_n(\cB))$.

We will denote the first component of $\mG_\mu$ with $G_\mu\colon\mathbb{H}^{+}({\cB})\lra \mathbb{H}^{-}({\cB})$. Whenever $\|b^{-1}\|<1/\|X\|$ we can write $G_X(b)=\sum_{n=0}^\infty b^{-1}\phi[(Xb^{-1})^n]$ as a convergent series. Thus, it follows easily that for $b\in\mathbb{H}^{+}(M_n(\cB))$,  we shall write,
 \[
 \mG_\mu(b)=\sum_{n=0}^\infty (\mu\otimes1_n)\bigl(b^{-1}[\X\cdot b^{-1}]^n\bigr)
 =(\mu\otimes1_n)\bigl([b-\X\cdot 1_n]^{-1}\bigr];
  \]
  (of course, these equalities require that we consider an extension of $\mu$ to $\mathcal B\langle\langle\mathcal X\rangle\rangle$, the algebra of formal power series generated freely by $\mathcal B$ and the selfadjoint symbol $\mathcal X$). This also indicates a very important equality, namely
\begin{equation}\label{G-frakH}
{\mathcal G}_\mu(b^{-1})=\mathfrak H_\mu(b),\quad b\in\mathbb H^+(M_n(\cB)).
\end{equation}
Moreover, ${\mathcal G}_\mu(b^*)=[{\mathcal G}_\mu(b)]^*$ extends ${\mathcal G}_\mu$ to the lower half-planes, analytically through points $b$ with inverse of small norm.

 Consequently, for $\mathcal F_X$, the reciprocal of ${\mathcal G}_X$, namely
\[ F_X(b)=[G_X(b)]^{-1}, \hspace{.6cm}
\mathcal F_X(b)=[\mathcal G_X(b)]^{-1}.\]
  we have that if $X, Y$ are monotone independent over $\cB$, then
  \[
  \mathcal F_{X+Y}(b)=\mathcal F_X\circ\mathcal F_Y(b).\]

 We would like also to mention the connection between $\mathcal F$ and $B$:
 \begin{equation}\label{F-B}
\mathds{1}-\mathcal F_\mu(b^{-1})b=B_\mu(b),\quad b^{-1}\in\mathbb H^+(\ncspace{\cB}).
\end{equation}
Henceforth, if the non-commutative random variables
$Y,X$ are Boolean independent with amalgamation over
$\mathcal B$, then, for $b\in
\mathbb H^+(\mathcal B_{nc})$ we have
 \[
\mathcal F_{X+Y}(b)-b=\mathcal F_X(b)-b+\mathcal F_Y(b)-b.
\]
 Finally, the {\em  $R$-transform of $X$} can be defined in terms of the Cauchy transform
as $R_X(b)={\mathcal G}^{-1}(b)-b^{-1}$ for any invertible $b$ so that $\|b\|$ is small.

\section{The operator-valued semicircular, Bernoulli and arcsine laws}

\subsection{The operator-valued Bernoulli law $\B$}

\begin{thm}\label{clthbool}

 Let $\{X_i\}_{i=1}^\infty$ be a family of centered, identically distributed (i.e. $\mu_{X_i}=\mu_(X_j)$ for all $i, j>0$), Boolean independent self-adjoint elements from a $\cB$-valued non-commutative C$^\ast$-probability space $(\cA, E_\cB)$.
 Consider
  \[ \eta:\cB\lra\cB,\  \eta(b)=E_\cB(X_i b X_i)\]
   the common variance of $X_i$'s. Then the law of $S_N=\frac{X_1+\dots X_N}{\sqrt{N}}$ converges weakly to an element of $\Sb^0$, that will be called $\B$, given by
 \begin{equation}\label{boolvar}
 B_\B(b)=\eta(b)\cdot b.
 \end{equation}

\end{thm}

\begin{proof}
 Since Boolean independence is preserved by tensoring with $M_n(\mathbb{C})$, we only need to prove the property for the first components of the $B$-transforms.

 From the equation (\ref{eqmb}), we have that
 $B_{X}(b)=\sum_{m=1}^\infty B_{X, m}(b)$, where $B_{X,m}(b)$ are given by the recurrences
 \[E_\cB([Xb]^m)=\sum_{k=1}^m B_{X, k}(b)E_\cB([Xb]^{m-k}).\]
 It follows that for any $\lambda\in\mathbb{R}$, $B_{\lambda X, m}(b)=\lambda^m B_{X, m}(b)$, so
 \[
 B_{S_N, m}(b)=\sum_{k=1}^N B_{\frac{X_k}{\sqrt{N}},m}(b)=N^{\frac{m-2}{2}}B_{X_i, m}(b).
 \]
 Therefore, if $m>2$, we have that
 $\lim_{N\to\infty}B_{S_N, m}(b)=0$, hence the conclusion.
\end{proof}

 Utilizing the result (\ref{boolvar}) in equations (\ref{eqmb}) and (\ref{F-B}) we obtain the following:

\begin{cor}
 With the above notations, we have that
 \begin{eqnarray}
 M_\B(b)&=&\left[\mathds{1}-\eta(b)b\right]^{-1}\\
 G_\B(b)&=&\left[b-\eta(b^{-1})\right]^{-1}\label{Gber}
 \end{eqnarray}
\end{cor}

\subsection{The operator-valued semicircular law $\gs$}

  The Central Limit Theorem law for free independence over $\cB$ have been described in \cite{V*}. We will just quote the result; the proof is analogue to the one of Theorem \ref{clthbool} above.

\begin{thm}

 Let $\{X_i\}_{i=1}^\infty$ be a family of centered, identically distributed, free independent self-adjoint elements from a $\cB$-valued non-commutative C$^\ast$-probability space $(\cA, E_\cB)$.
 Consider
  \[ \eta:\cB\lra\cB,\  \eta(b)=E_\cB(X_i b X_i)\]
   the common variance of $X_i$'s. Then the law of $\frac{X_1+\dots X_N}{\sqrt{N}}$ converges weakly to an element of $\Sb^0$, that will be called $\gs$, given by
 \[R_\gs(b)=\eta(b)\cdot b.\]

\end{thm}

  Using the equations (\ref{eqmr}), (\ref{eqmb}) and the definition of $R_X$ in terms of $\mathcal G_X$, the above theorem gives the following two analytic characterizations of $\gs$, the second also shown in \cite{HRS}:

   \begin{cor}\label{cor2.4}
 With the above notations, we have that
 \begin{eqnarray}
 B_\gs(b)&=&\eta(M_\gs(b))\cdot b\\
 b&=&F_\mathfrak s(b)+\eta(G_\mathfrak s(b)),\quad b\in\mathbb H^+(B).\label{S}
 \end{eqnarray}

 \end{cor}

  A combinatorial, more explicit, description of $\gs$ is done R. Speicher (\cite{RS2}) in terms of non-crossing pair partitions. We cite the result, with the notations from the present material, in the next section (Proposition \ref{comb-semicirc}).

\subsection{The operator-valued arcsine law $\ga$}

 The general description of $\ga$ will be made in combinatorial terms.

  First we need some  notations. $NC(n)$ will denote the set of all non-crossing partitions on an ordered set with $n$ elements (we can identify it notationally with $\langle n \rangle =\{1,2,\dots, n\}$). $NC_2(n)$ will denote the subset of $NC(n)$ with the property that all their blocks contain exactly 2 elements; if $n$ is odd, then $NC_2(n)$ is the void set.

  For $\gamma\in NC_2(n)$, we will denote by $\widetilde{\gamma}$ the element in $NC_2(n+2)$ such that $(1,n+2)$ is a block in $\widetilde{\gamma}$ and $\widetilde{\gamma}\setminus \{(1, n+2)\}\cong \gamma$. For example, if $\gamma=\{(1,4), (2,3), (5,6)\}\in NC_2(6)$, then $\widetilde{\gamma}=\{ (1,8), (2,5), (3,4), (6,7)\}$. Also, if $\gamma_1\in NC_2(n)$ and $\gamma_2\in NC_2(m)$, we will denote by $\gamma_1\oplus\gamma_2$ the element of $NC_2(n+m)$ obtained by juxtaposing $\gamma_1$ and $\gamma_2$ in this order. Finally, if $\pi=(B_1, \dots, B_q)\in NC(n)$, then we denote by $\textgoth{F}(\pi)$ the set of all bijections  from $\{B_1, \dots, B_q\}$ to $\{1, \dots, q\}.$

 The next Theorem will give a combinatoric description of the Central Limit theorem law for monotone independence, refining the result from \cite{P}.

 \begin{thm}\label{ga}
  Let $\{X_i\}_{i=1}^\infty$ be a family of centered, identically distributed, monotone independent self-adjoint elements from a $\cB$-valued non-commutative C$^\ast$-probability space $(\cA, E_\cB)$.

  Denote by $S_N=\frac{X_1+\dots X_N}{\sqrt{N}}$, $\sigma_N=\mu_{S_N}$ and by
  \[ \eta:\cB\lra\cB,\  \eta(b)=E_\cB(X_i b X_i)\]
   the common variance of $X_i$'s.

   With the above notations, the sequence of conditional expectations $\sigma_N$ converges weakly to a conditional expectation $\mathfrak{a}$ which depends only on $\eta$ and its fully matrical extension $\widetilde{\ga}$ is described by
  \[\widetilde{\ga}\left( [Xb]^n\right)=\sum_{\gamma\in NC_2(n)} V(\gamma, b)\cdot b
  \]
  where $V(\gamma,b)$ are given by the following recurrences:
  \begin{enumerate}
  \item[(1a)] $V((1,2),b)= \eta(b)$
  \item[(2a)] $V(\gamma_1\oplus\gamma_2, b)=V(\gamma_1,b)\cdot b \cdot V(\gamma_2, b)$
  \item[(3a)]$V(\widetilde{\gamma})=\frac{1}{|\gamma|+1}\eta(b\cdot V(\gamma, b)\cdot b)$.
  \end{enumerate}

 \end{thm}

 \begin{proof}
    Since, as also stated above, the montoone independence is preserved under tensoring with ${M_n(\mathbb{C})}$, it suffices the prove the result for $\mathfrak{a}$.  Also, eventually replacing each $X_j$ with $X_jb$, we can suppose that $b=1$, henceforth we need to compute
    \begin{eqnarray*}
    m_n&=&
    \lim_{N\lra\infty}E_\cB\left(\Bigl(\frac{X_1+\dots + X_N}{\sqrt{N}}\Bigr)^n\right)\\
    &=&
    \lim_{N\lra\infty}
    \sum_{\substack{1\leq \epsilon_j\leq N\\ 1\leq j\leq n}}\frac{1}{N^{\frac{n}{2}}}\cdot E_\cB\bigl(X_{\epsilon_1} \cdots X_{\epsilon_n}\bigr).
    \end{eqnarray*}
      To each $\overrightarrow{\epsilon}=(\epsilon_1, \dots, \epsilon_n)\in \langle N\rangle^n$ we associate a pair $(\pi_{\overrightarrow{\epsilon}}, f_{\overrightarrow{\epsilon}})\in NC(N)\times\textgoth{F}(NC(N))$ by putting all $\epsilon_j$'s that are equal in the same block and, for $B\in \pi_{\overrightarrow{\epsilon}}$, defining $f_{\overrightarrow{\epsilon}}(B)=s$ if there are exactly $s-1$ blocks in $\pi_{\overrightarrow{\epsilon}}$ containing $\epsilon_j$'s smaller than the ones in $B$. Note that if $(\pi_{\overrightarrow{\epsilon}}, f_{\overrightarrow{\epsilon}})=(\pi_{\overrightarrow{\epsilon^\prime}}, f_{\overrightarrow{\epsilon^\prime}})$,  then
  \[
  E_\cB(X_{\epsilon_1}\cdots X_{\epsilon_n})=E_\cB(X_{\epsilon^\prime_1}\cdots X_{\epsilon^\prime_n})= V(\pi_{\overrightarrow{\epsilon}}, f_{\overrightarrow{\epsilon}}).
  \]
    From the relations defining the monotone independence, if there exists some $j$ with $\epsilon_j\neq \epsilon_k$ if $j\neq k$, then $E_\cB(X_{\epsilon_1} \cdots X_{\epsilon_n})=0$. I.e. if $\pi$ has block with only one element, then $V(\pi_{\overrightarrow{\epsilon}}, f_{\overrightarrow{\epsilon}})=0$ for all $\overrightarrow{\epsilon}$ with $\pi_{\overrightarrow{\epsilon}}=\pi$.  Particularly, for $n=2$, the limit is
    \[m_2=
    \lim_{N\lra\infty}\sum_{j=1}^N\frac{1}{N}E_\cB(X_j^2)=E_\cB(1).\]
 so the relation (1a) is proved.

 Denoting by $NC^\prime(n)$ the set of all $\pi\in NC(n)$ with each of their blocks containing at least two elements and using the above notations, we have:
  \begin{eqnarray*}
  m_n
  &=&
  \sum_{\pi\in NC^\prime(n)}\lim_{N\lra\infty}\frac{1}{N^{\frac{n}{2}}}
  \sum_{\substack{\overrightarrow{\epsilon}\in\langle N\rangle\\\pi_{\overrightarrow{\epsilon}} = \pi }} V(\pi,f_{\overrightarrow{\epsilon}})\\
  &\leq&
  \sum_{\pi\in NC(n)}\lim_{N\lra\infty}\frac{1}{N^{\frac{n}{2}}}\cdot N^{|\pi|}\max_{f\in\textgoth{F}(\pi)}V(\pi,f).
  \end{eqnarray*}
   Since $\pi\in NC^\prime(n)$, we have that $|\pi|<\frac{n}{2}$ and the limit is 0,  unless $\pi\in NC_2(n)$, hence
   \begin{equation}\label{mneq}
   m_n=\sum_{\pi\in NC^\prime(n)}\lim_{N\lra\infty}\frac{1}{N^{\frac{n}{2}}}\cdot\sum_{\substack{\overrightarrow{\epsilon}\in\langle N\rangle\\\pi_{\overrightarrow{\epsilon}} = \pi }} V(\pi,f_{\overrightarrow{\epsilon}})
   \end{equation}

\noindent   With the notation $\displaystyle V_n(\pi)=\sum_{\substack{\overrightarrow{\epsilon}\in\langle N\rangle^n\\\pi_{\overrightarrow{\epsilon}} = \pi }} V(\pi,f_{\overrightarrow{\epsilon}})$, if suffices to prove that
   \[
   \lim_{N\lra\infty}\frac{1}{N^{\frac{n}{2}}}V_N(\pi)=V(\pi) \]
   exists for all $\pi\in NC_2(n)$ and satisfies (2) and (3).

   For (2a), note first that
   \begin{equation}\label{oplus}
   V_N(\pi_1\oplus \pi_2)=V_N(\pi_1)\cdot V_{N-|\pi_1|}(\pi_2)
   \end{equation}
 Indeed, if ${\overrightarrow{\epsilon}}$ is such that $\pi_{\overrightarrow{\epsilon}}=\pi_1\oplus\pi_2$, then it is the concatenation of some  ${\overrightarrow{\epsilon_1}}$  and ${\overrightarrow{\epsilon_2}}$
 with disjoint set al components such that $\pi_{\overrightarrow{\epsilon_1}}=\pi_1$ and $\pi_{\overrightarrow{\epsilon_2}}=\pi_2$. Choosing the components of $\overrightarrow{\epsilon}$ from $\langle N \rangle$ can be seen as first choosing the components of $\overrightarrow{\epsilon_1}$ from $\langle N \rangle$, then choosing the components of $\overrightarrow{\epsilon_2}$ from the remaining $N-|\pi_1|$ posibilities, hence (\ref{oplus}).

  It follows that
  \begin{eqnarray*}
  V(\pi_1\oplus\pi_2)&=&
  \lim_{N\lra\infty}\frac{1}{N^{|pi_1|+|\pi_2|}}V_N(\pi_1\oplus\pi_2)\\
  &=&
  \lim_{N\lra\infty}\frac{1}{N^{|pi_1|}}V_N(\pi_1)\cdot \frac{1}{N^{|\pi_2|}}V_{N-|\pi_1|}(\pi_2)\\
  &=&
  V(\pi_1)\cdot\lim_{N\lra\infty}\frac{(N-|\pi_1|)^{|\pi_2|}}{N^{|\pi_2|}}
  \cdot\frac{1}{(N-|\pi_1|)^{|\pi_2|}} V_{N-|\pi_1|}(\pi_2)\\
  &=& V(\pi_1)\cdot V(\pi_2), \ \text{hence (2)}.
  \end{eqnarray*}

  For (3a), note first that if $\pi\in NC_2(n)$ and $\overrightarrow{\epsilon}=(\epsilon_1, \dots, \epsilon_{n+2})$ is such that $\pi_{\overrightarrow{\epsilon}}=\widetilde{\pi}$, then $V_{(\widetilde{\pi}, f_{\overrightarrow{\epsilon}})}=0$ unless $\epsilon_1=\epsilon_{n+2}<\epsilon_j$ for all $j=2,\dots, n-1$.

   Indeed, if the smallest components of $\overrightarrow{\epsilon}$ are some $\epsilon_j, \epsilon_l$ with $1<j<l<n+2$, then, from the relations defining the monotone independence, we have that
   \begin{eqnarray*}
   V_{(\widetilde{\pi}, f_{\overrightarrow{\epsilon}})}
   &=&
   E_\cB(X_{\epsilon_1}\cdots X_{\epsilon_{n+2}})\\
   &=&
   E_\cB(X_{\epsilon_1}\cdots X_{\epsilon_j}
   E_\cB( X_{\epsilon_j+1}\cdots X_{\epsilon_l-1})
   X_{\epsilon_l}\cdot X_{\epsilon_l+1}\cdots X_{n+2})\\
   &&\hspace{-1.5cm}=E_\cB(X_{\epsilon_1}\cdots X_{\epsilon_j-1})\cdot
   E_\cB(X_{\epsilon_j}
   E_\cB( X_{\epsilon_j+1}\cdots X_{\epsilon_l-1})
   X_{\epsilon_l})\cdot E_\cB(X_{\epsilon_l+1}\cdots X_{n+2}).
   \end{eqnarray*}
  The set $\{\epsilon_1, \dots, \epsilon_j-1 \}$ does not have any other elements equal to $\epsilon_1$ therefore $E_\cB(X_{\epsilon_1}\cdots X_{\epsilon_j-1})=0$ hence $V_{(\widetilde{\pi}, f_{\overrightarrow{\epsilon}})}$ cancels.

  Moreover, if $\epsilon_1=\epsilon_{n+2}<\epsilon_j$ for all $j=2,\dots, n-1$, then the monotone independence gives
  \begin{eqnarray}
  V_{(\widetilde{\pi}, f_{\overrightarrow{\epsilon}})}
   &=&
   E_\cB(X_{\epsilon_1}\cdot E_\cB(X_{\epsilon_2}\cdots X_{\epsilon_{n+1}})\cdot X_{\epsilon_{n+2}})\nonumber\\
   &=&\eta(V_{(\pi, f_{(\epsilon_2, \dots, \epsilon_{n+1}})})\label{tilde}
  \end{eqnarray}

  Next we will prove (3a) together with the following relation: that for all $\pi\in NC_2(n)$ ($n\geq2$, even) and all $N\geq n$ we have that
  \begin{equation}\label{inductie}
  V_N(\pi)=V(\pi)\cdot N^{\frac{n}{2}}+P_\pi(N)
  \end{equation}
  where $P_\pi$ is a polynomial of degree at most $\frac{n}{2}-1$. Remark that (\ref{inductie}) implies the existence of $V(\pi)$.

  For $n=2$, the relation (\ref{inductie}) is equivalent to (1). Suppose now (\ref{inductie}) true for $n\leq 2m$ and fix $\pi\in NC_2(2m+2)$. Then $\pi$ is either of the form $\pi_1\oplus\pi_2$ for some non-crossing pair partitions $\pi_1$ and $\pi_2$ with $|\pi_1|+|\pi_2|=m+1$ or of the form $\widetilde{\sigma}$ for some $\sigma\in NC_2(2m)$.

  If $\pi=\pi_1\oplus\pi_2$ then the equation (\ref{oplus}) gives
  \begin{eqnarray*}
  V_N(\pi)&=&V_N(\pi_1)\cdot V_{N-|\pi_1|}(\pi_2)\\
  &=&
  \bigl(V(\pi_1)\cdot N^{|\pi_1|}+P_{\pi_1}(N) \bigr)\cdot
  \bigl(V(\pi_2)\cdot (N-|\pi_1|)^{|\pi_2|}+P_{\pi_2}(N-|\pi_1|)\bigr)\\
  &=&
  [V(\pi_1)V(\pi_2)]\cdot N^{|\pi_1|+|\pi_2|}+P_{\pi}(N)
  \end{eqnarray*}
  and the conclusion follows from (2a).

  If $\pi=\widetilde{\sigma}$ for some $\sigma\in NC_2(2m)$, the definition of $V_N(\pi)$  is
  \begin{equation}\label{vn}
  V_N(\pi)=\sum_{\substack{\overrightarrow{\epsilon}\in\langle N\rangle^{2m+2}\\\pi_{\overrightarrow{\epsilon}}=\pi}} V(\pi, f_{\overrightarrow{\epsilon}}),
  \end{equation}
  but, as seen above, the terms $V(\pi, f_{\overrightarrow{\epsilon}})$ cancel unless $\overrightarrow{\epsilon}=(l, \eta_1, \dots \eta_{2m+1}, l)$, with $l$ smaller than all $\eta_j$ (henceforth $l<N-m$) and $\pi_{\overrightarrow{\eta}}=\sigma$ for $\overrightarrow{\eta}=(\eta_1, \dots, \eta_{2m+1})$. Also, the ordered set $\langle N\rangle_l=\{l+1, \dots, N\}$ is isomorphic to $\langle N-l\rangle$, therefore the equality (\ref{vn}) becomes
  \begin{eqnarray*}
  V_N(\pi)&=&
  \sum_{l=1}^{N-m}
  \sum_{\substack{\overrightarrow{\eta}\in \langle N\rangle_l^{2m}\\ \pi_{\overrightarrow{\eta}}=\sigma}}
  V(\pi, f_{(l, \eta_1, \dots, \eta_{2m+1},l)})\\
  &=&
 \sum_{l=1}^{N-m}
 \sum_{\substack{\overrightarrow{\eta}\in \langle N-l\rangle^{2m}\\ \pi_{\overrightarrow{\eta}}=\sigma}}
  \eta\bigl(V(\sigma, f_{\overrightarrow{\eta}})\bigr)
  \end{eqnarray*}
  where for the last equality we used the argument above and equation (\ref{tilde}). It follows that
  \begin{eqnarray}
  V_N(\pi)&=&\sum_{l=1}^{N-m}\eta\bigl(\sum_{\substack{\overrightarrow{\eta}\in \langle N-l\rangle^{2m}\\ \pi_{\overrightarrow{\eta}}=\sigma}}
  V(\sigma, f_{\overrightarrow{\eta}})\bigr) \nonumber\\
  &=&\sum_{l=1}^{N-m}\eta\bigl( V_{N-l}(\sigma)\bigr)
  \nonumber\\
  &=&\eta\bigl(\sum_{l=1}^{N-m}V_{N-l}(\sigma)\bigr).\label{3star}
  \end{eqnarray}
  From de induction hypothesis, equation (\ref{3star}) is equivalent to
  \begin{eqnarray*}
  V_N(\pi)&=&
  \eta\bigl(\sum_{l=1}^{N-m}[V(\sigma)\cdot (N-l)^m+P_\sigma(N-l)]\bigr)\\
  &=&\eta\bigl(V(\sigma))\cdot[\sum_{l=1}^{N-m}(N-l)^m]\bigr) + Q_\sigma(N)
  \end{eqnarray*}
  where $Q_\sigma$ is a polynomial of degree at most $m$. The proof for (3a) and (\ref{inductie}) is now finished by noticing that, from the well-known approximation with Riemann sums of $\int_0^1x^m dx$, the coefficient of $N^{m+1}$ in
  $
  \sum_{l=1}^{N-m}(N-l)^m$ is $\frac{1}{m+1}$.
 \end{proof}

 \begin{cor}\label{monstable}
  $\ga$ is stable with respect to the monotone convolution. More precisely, if $\ga_2$ is the dilation with $\sqrt{2}$ of $\ga$, then
  \[ \gH_\ga\circ\gH_\ga=\gH_{\ga_2}.\]
\end{cor}

\begin{proof}
  Let $a_1, b_1, a_2, b_2, ....$ be a sequence of centered, monotone independent non-commutative random variables of variance $\eta$ and $X_i=a_i+b_i$. It follows that $\{X_i\}_i$ are also monotone independent, centered and of variance $2\eta$, hence
  $S_{2N}=\frac{X_1+\dots X_{2N}}{\sqrt{2N}}$ will converge in distribution to $\ga_2$, but $S_{2N}=r_N+t_N$, were
  $r_N=\frac{a_1+b_1+\dots +a_{N}+b_{N}}{\sqrt{2N}}$ and $t_N=\frac{a_{N+1}+b_{N+1}+\dots +a_{2N}+b_{2N}}{\sqrt{2N}}$;
  For all $N$ we have that $r_N$ and $t_N$ are monotone independent, and they converge to $\ga$; the conclusion follows from the remark that the $n$-th moment of $\gH_\ga\circ\gH_\ga$ depends only on the first $n$ moments of $\ga$.
\end{proof}

 \begin{prop}\label{propaj} Denote $a_{n}(b)=\widetilde{\ga}([\X b]^n )$. The $B$-transforms of $\ga$  satisfies the following relation:
\[B_\ga(b)=\sum_{n=0}^\infty \frac{1}{n}\eta(b\cdot a_{2n-2}(b)) \cdot b\]
\end{prop}

  \begin{proof}
 Theorem \ref{ga} gives:
 \begin{eqnarray*}
 a_n(b)
 &=&
\sum_{\pi\in NC_2(n)}V(\pi,b)b\\
 && \hspace{-1cm}=
[\sum_{\pi\in NC_2(n-2)}V(\widetilde{\pi},b)b] + [\sum_{p=2}^{n-2}\bigl( \sum_{\pi\in NC_2(p-2)} V(\widetilde{\pi}, b)b\cdot \sum_{\gamma\in NC_2(n-(p+2)}V(\gamma, b)b\bigr)]\\
&=&\sum_{\pi\in NC_2(n)} \frac{1}{\frac{n}{2}}\eta(b\cdot V(\pi, b)\cdot b)b\\
&&+\sum_{p=2}^{n-2}[ \sum_{\pi\in NC_2(p-2)}\frac{1}{\frac{p}{2}}\eta(b\cdot V(\pi_1,b)\cdot b)b\cdot \sum_{\gamma\in NC_2(n-(p+2)}V(\gamma, b)b]\\
&=&
\frac{1}{\frac{n}{2}}\eta(b\cdot a_{n-2}(b)\cdot b) + \sum_{p=2}^{n-2}[\frac{1}{\frac{p}{2}}\eta(b\cdot a_{p-2}(b)\cdot b)\cdot a_{n-(p+2)}(b)]
 \end{eqnarray*}
 Comparing the above relation to the recurrence for the $B$-transform, we have that
 \[B_{2n, \ga}(b)=\frac{1}{n}\eta(b\cdot m_{2n-2}(b))\cdot b\]
and all the coefficients of $B_\ga$ of odd order are 0, we conclude.
\end{proof}

 \begin{cor}\label{rmk4}
 If $\ga$ is given by the variance $\eta(b)=aba$ for some self-adjoint $a\in\cB$, then the Cauchy transform of $\ga$ satisfies:
 \[ \left(bG_\ga(b)\right)^2 =1+4[aG_\ga(b)]^2
 \]
\end{cor}

\begin{proof}
  If $\cB=\mathbb{C}$, we have that $\ga=\ba$, the classical arcsine law and
 $[zG_\ba(z)]^2=1+4G_\ba(z)^2$. With the notation $\alpha_n$ for the $n$-th moment of $\ba$, the equation becomes
 \begin{equation}\label{provizoriu1}
 \left(\sum_{k=0}^\infty\frac{\alpha_k}{z^{k}}\right)^2=1+4 \left(\sum_{k=0}^\infty\frac{\alpha_p}{z^{p+1}}\right)^2
 \end{equation}

Identifying the coefficients of $z^{-n}$ in both sides of (\ref{provizoriu1}), we obtain that for all $n\geq 1$
\begin{equation}\label{alpharel}
\sum_{p=0}^n \alpha_p\alpha_{n-p}=4\sum_{l=0}^{n-2}\alpha_l\alpha_{n-2-l}
\end{equation}

Moreover, we have that
\begin{equation}\label{alpharel2}
\alpha_n=\sum_{pi\in NC_2(n)}V_\ba(\pi)\ \text{where $V_\ba(\pi)\in\mathbb{R}$ are satisfying (1a)-(3a)}
\end{equation}

 If $\ga$ is given by $\eta(b)=aba$ for some self-adjoint $a\in\cB$, then (1a)-(3a) imply that $V_\ga(\pi, b)=v_\ga(\pi)(ab)^{n-1}a$ for some $v_\ga(\pi)\in\mathbb{R}$. It is easy to see that $v_\ga(\pi)$ also satisfy (1a)-(3a), henceforth $v_\ga(\pi)=V_\ba(\pi)$ and (\ref{alpharel2}) implies
\[
\ga((Xb)^n)=\sum_{\pi\in NC_2(n)}V_\ba(\pi)(ab)^n=\alpha_n (ab)^n
\]
and
\begin{eqnarray*}
G_\ga(b)&=&\varphi\bigl(b^{-1}[1-Xb^{-1}]^{-1}\bigr)\\
&=&b^{-1}\sum_{k=0}^\infty\varphi\bigl((Xb^{-1})^k\bigr)=b^{-1}\sum_{k=0}^\infty\alpha_k (ab^{-1})^k.\\
\end{eqnarray*}
Henceforth
\begin{eqnarray*}
\left(bG_\ga(b)\right)^2
&=&
\left[ \sum_{k=0}^\infty \alpha_k(ab^{-1})k\right]^2\\
&=&
\sum_{n=0}^\infty\left(\sum_{k=0}^n\alpha_k\alpha_{n-k}\right) (ab^{-1})^n\\
&=&
1+\sum_{n=2}^{\infty}\left[\sum_{l=0}^{n-2}\alpha_l\alpha_{n-l}\right](ab^{-1})^n\\
&=&
1+\sum_{n=0}^\infty\left(\sum_{l=0}^n a\alpha_l b^{-1}(ab^{-1})^l\cdot a\alpha_{n-l}b^{-1}(ab^{-1})^{n-l}\right)\\
&=&
1+4[aG_\ga(b)]^2
\end{eqnarray*}\
\end{proof}

 In \cite{RS2} a similar combinatorial treatment is done for the operator-valued semicircular law $\gs$; with the above notations, $\gs$ can be combinatorially described as follows:

\begin{prop}\label{comb-semicirc}
The op-valued free central limit law $\gs$ of variance $\eta$ (that is $\eta:\cB\lra\cB$ is the map $b\mapsto \gs(\X b\X)$) is combinatorially described by
 \[
 \widetilde{\gs}\left( [Xb]^n\right)=\sum_{\gamma\in NC_2(n)} W(\gamma, b)\cdot b
  \]
  where $W(\gamma,b)$ are given by the following recurrences:
  \begin{enumerate}
  \item[(1s)] $W((1,2),b)= \eta(b)$
  \item[(2s)] $W(\gamma_1\oplus\gamma_2, b)=W(\gamma_1,b)\cdot b \cdot W(\gamma_2, b)$
  \item[(3s)]$W(\widetilde{\gamma})=\eta(b\cdot V(\gamma, b)\cdot b)$.
  \end{enumerate}
\end{prop}

\section{Relations between operator-valued Bernoulli, arcsine and semicircular distributions}

As mentioned in the introduction, in scalar-valued noncommutative probability the free additive convolution of two Bernoulli distributions as well as the Boolean convolution of two semicircular distributions is the arcsine distribution. In this section we shall make explicit to what extent this connection holds for operator-valued distributions.

It has been shown in \cite{BPV1} that the Boolean-to-free Bercovici-Pata bijection sends $\B$ to $\mathfrak s$. One of the important results of Voiculescu used in the proof of this result is the subordination property for free convolution \cite{V2}: if $X$ and $Y$ are free over $\mathcal B$, then there exists $\omega\colon\mathbb H^+(\mathcal B)\to\mathbb H^+(\mathcal B)$ analytic so that $G_{\mu_{X+Y}}(b)=G_{\mu_X}(\omega(b))$, $b\in\mathbb H^+(\mathcal B)$. This relation holds for the corresponding fully matricial extensions.

We remind the reader one of the tools used for proving the Boolean-to-free Bercovici-Pata bijection, namely \cite[Proposition 3.1]{BPV1}:

\begin{prop}\label{5}
For any $\mathcal B$-valued distribution $\mu$, we
denote $\omega_n$ the subordination function for
$\mu^{\boxplus n}=\underbrace{\mu\boxplus\mu\boxplus\cdots\boxplus\mu}_{n\textrm{ times}}$.
Then the following
functional equations hold {\em:
\begin{equation}\label{o}
\omega_n(b)=\frac1nb+\left(1-\frac1n\right)F_{\mu^{\boxplus n}}(b)=\frac1nb+\left(1-\frac1n\right)F_{\mu}(\omega_n(b)),
\end{equation}
\begin{equation}\label{F}
F_{\mu^{\boxplus n}}(b)=F_\mu\left(\frac1nb+\left(1-\frac1n\right)F_{\mu^{\boxplus n}}(b)\right),\quad b\in \mathbb H^+(M_n(\mathcal B)).
\end{equation}
}
\end{prop}

Our first result of this section is the following.

\begin{thm}\label{4}
Assume that $a\in\mathcal B^{\rm sa}$ and $\B$ is concentrated in the points $-a$ and $a$ (i.e. $\B(\mathcal Xb\mathcal X)=aba$). Then $\B\boxplus\B=\mathfrak a$, where $\mathfrak a$ is the centered arcsine distribution of variance $b\mapsto2aba,$ $b\in\mathcal B$. In addition, if $a$ is invertible, then the reciprocal of the Cauchy transform of $\mathfrak a$ satisfies the Abel equation
$$
\phi(F_\mathfrak a(b))=\phi(b)-4,\quad b\in\mathbb H^+(\mathcal B),
$$
where $\phi(w)=wa^{-1}wa^{-1},$ $w\in\mathcal B$. All relations extend to $\mathcal B_{nc}$.
\end{thm}

\begin{proof}
It follows from its definition that $F_\B(b)=b-ab^{-1}a,$ for invertible $b\in\mathcal B$. (Sometimes it will be more convenient to view this relation in the form $F_\B(b)=(b+a)b^{-1}(b-a)$.) Assume for the beginning that $a$ is invertible. We claim that $F_\B(b)a^{-1}b=ba^{-1}F_\B(b),$ $b\in\mathbb H^+(\mathcal B)$. Indeed,
$$
F_\B(b)a^{-1}b=(b-ab^{-1}a)a^{-1}b=ba^{-1}b-a
$$
and
$$
ba^{-1}F_\B(b)=ba^{-1}(b-ab^{-1}a)=ba^{-1}b-a,
$$
from which we conclude.

Let us first note a few obvious properties of the transforms involved: first, if $\omega_t\colon\mathbb H^+(\mathcal B)\to\mathbb H^+(\mathcal B)$ satisfies $t\omega_t(b)=b+(t-1)F_\B(\omega_t(b))$,
then we can apply the above observation for $b$ replaced by $\omega_t(b)$ to conclude that
\begin{eqnarray*}
ba^{-1}\omega_t(b) & = & t\omega_t(b)a^{-1}\omega_t(b)-(t-1)F_\B(\omega_t(b))a^{-1}\omega_t(b)\\
& = & t\omega_t(b)a^{-1}\omega_t(b)-(t-1)\omega_t(b)a^{-1}F_\B(\omega_t(b))\\
& = & \omega_t(b)a^{-1}b,\quad b\in\mathbb H^+(\mathcal B).
\end{eqnarray*}
Applying this to $t=n\in\mathbb N$ we obtain the equality
\begin{equation}\label{commut-omega}
\omega_n(b)a^{-1}b=ba^{-1}\omega_n(b),\quad b\in\mathbb H^+(\mathcal B),
\end{equation}
for the omega function from \eqref{o}. Since $F_{\B^{\boxplus n}}(b)=
\frac{n}{n-1}\omega_n(b)-\frac{1}{n-1}b$ and $(ba^{-1})b=b(a^{-1}b)$, we also conclude that
\begin{equation}\label{commut-FBer}
F_{\B^{\boxplus n}}(b)a^{-1}b=ba^{-1}F_{\B^{\boxplus n}}(b),\quad b\in\mathbb H^+(\mathcal B),n\in\mathbb N.
\end{equation}
We shall consider this relation particularly for $n=2$. Writing relation \eqref{F} for $n=2$ and $\mu=\B$ gives
$$
F_{\B\boxplus\B}(b)=\frac12(b+F_{\B\boxplus\B}(b))-2a(b+F_{\B\boxplus\B}(b))^{-1}a.
$$
We simplify and multiply left with $(b+F_{\B\boxplus\B}(b))a^{-1}$ to get
$$
ba^{-1}F_{\B\boxplus\B}(b)+F_{\B\boxplus\B}(b)a^{-1}F_{\B\boxplus\B}(b)=ba^{-1}b+F_{\B\boxplus\B}(b)a^{-1}b-4a.
$$
We simplify according to equation \eqref{commut-FBer} and multiply with $a^{-1}$ to the right to obtain
\begin{equation}
F_{\B\boxplus\B}(b)a^{-1}F_{\B\boxplus\B}(b)a^{-1}=ba^{-1}ba^{-1}-4.
\end{equation}
We observe from this relation that
\begin{eqnarray*}
\lefteqn{F_{\B\boxplus\B}(F_{\B\boxplus\B}(b))a^{-1}F_{\B\boxplus\B}(F_{\B\boxplus\B}(b))a^{-1} =}\\
& & F_{\B\boxplus\B}(b)a^{-1}F_{\B\boxplus\B}(b)a^{-1}-4= \\
& & ba^{-1}ba^{-1}-8.
\end{eqnarray*}
Generally,
\begin{equation}\label{Abel}
F_{\B\boxplus\B}^{\circ n}(b)a^{-1}F_{\B\boxplus\B}^{\circ n}(b)a^{-1}=ba^{-1}ba^{-1}-4n,\quad b\in\mathbb H^+(\mathcal B),n\in\mathbb N.
\end{equation}
This is the operator-valued version of Abel's equation $\phi(F(b))=\phi(b)+c$, with $\phi(b)=ba^{-1}ba^{-1}$, $F=F_{\B\boxplus\B}$ and $c=-4\cdot{\bf 1}_\mathcal B$. On the other hand,
$$
\sqrt{n}F_{\B\boxplus\B}(b/\sqrt{n})a^{-1}\sqrt{n}F_{\B\boxplus\B}(b/\sqrt{n})a^{-1}={b}a^{-1}{b}a^{-1}-4n,
$$
so $b\mapsto\sqrt{n}F_{\B\boxplus\B}(b/\sqrt{n})$ and $b\mapsto F_{\B\boxplus\B}^{\circ n}(b)$ satisfy exactly the same functional equations and they both map the fully matricial upper half-plane into itself; by analyticity, they must coincide: $\sqrt{n}F_{\B\boxplus\B}(b/\sqrt{n})=F_{\B\boxplus\B}^{\circ n}(b)$ for all $b$ with positive imaginary part. Moreover, re-normalizing in equation \eqref{Abel} and taking limit gives us
\begin{eqnarray*}
F_\mathfrak a(b)a^{-1}F_\mathfrak a(b)a^{-1} & = & \lim_{n\to\infty}
\frac{F_{\B\boxplus\B}^{\circ n}(\sqrt{n}b)}{\sqrt{n}}a^{-1}\frac{F_{\B\boxplus\B}^{\circ n}(\sqrt{n}b)}{\sqrt{n}}a^{-1}\\
& = & \lim_{n\to\infty}\frac{1}{n}(\sqrt{n}ba^{-1}\sqrt{n}ba^{-1}-4n)\\
& = & {b}a^{-1}{b}a^{-1}-4,
\end{eqnarray*}
according to the monotonic central limit proved in \cite{P} (see also Theorem \ref{ga} above).
Thus, $F_\mathfrak a$ satisfies also the same functional equation as $F_{\B\boxplus\B}$. We conclude that
$$
\frac{F_{\B\boxplus\B}^{\circ n}(\sqrt{n}b)}{\sqrt{n}}=F_{\B\boxplus\B}(b)=F_\mathfrak a(b)=\frac{F_{\mathfrak a}^{\circ n}(\sqrt{n}b)}{\sqrt{n}},\quad b\in\mathbb H^+(\mathcal B).
$$
This proves our proposition for the case when $a$ is invertible in $\mathcal B$. The general case follows now easily: we approximate $a$ with $a_\varepsilon=f_\varepsilon(a)$, where $f_\varepsilon\colon\mathbb R\to\mathbb R$ is defined by
$$f_\varepsilon(x)=\left\{\begin{array}{lcl}
x & \textrm{if} & |x|\ge\varepsilon\\
\varepsilon & \textrm{if} & |x|<\varepsilon
\end{array}\right.
$$
Then recalling that $F_\B(b)=b-ab^{-1}a$,
\begin{eqnarray*}
\lim_{\varepsilon\to0}\|b-ab^{-1}a-b+a_\varepsilon b^{-1}a_\varepsilon\| & = & \lim_{\varepsilon\to0}\|(a-a_\varepsilon)b^{-1}a_\varepsilon+ab^{-1}(a-a_\varepsilon)\|\\
& \le & \lim_{\varepsilon\to0}\|a-a_\varepsilon\|\|b^{-1}\|(\|a_\varepsilon\|+\|a\|)\\
& \le & 2\|a\|\|b^{-1}\|\lim_{\varepsilon\to0}\|a-a_\varepsilon\|\\
& = & 2\|a\|\|b^{-1}\|\lim_{\varepsilon\to0}\sup_{x\in\mathbb R}|f(x)-x|\\
& = & 4\|a\|\|b^{-1}\|\lim_{\varepsilon\to0}\varepsilon\\
& = & 0
\end{eqnarray*}
The limit is uniform for $b$ in closed balls included in $\mathbb H^+(\mathcal B)$. This shows that $\B\boxplus\B=\mathfrak a$ whenever the variance of $\B$ is $b\mapsto aba$.
\end{proof}

It has been shown in Corollary \ref{rmk4} that the operator-valued Cauchy transform of $\mathfrak a$ of variance $b\mapsto aba$ is characterized by the equation
$$
\left[bG_\mathfrak a(b)\right]^2=1+4\left[aG_\mathfrak a(b)\right]^2,\quad b\in\mathbb H^+(\mathcal B).
$$
We note that the above proof provides us also with an argument proving exactly the same result.
When $b=z1_\cB$,
$$
(z1_\cB-2a)G_\mathfrak a(z1_\cB)(z1_\cB+2a)G_\mathfrak a(z1_\cB)=1_\cB, \quad z\in\mathbb C^+;
$$
by choosing $a=\alpha 1_\mathcal B$ we easily find here the equation of the Cauchy transform of the classical usual arcsine distribution
$G_\mathfrak a (z)=[z^2-4\alpha]^{-\frac12}.$

Of course, only ``few'' of all completely positive maps are of the form $b\mapsto aba$. Using Voiculescu's theory of fully matricial maps and Stinespring's theorem, the next proposition shows that nevertheless understanding arcsine distributions with variances of the above form covers many cases of interest.
\begin{prop}\label{prop8}
Let {\em $\B$} be the Bernoulli concentrated in $-a,a\in\cB$, and {\em $\B\boxplus\B=\mathfrak a$.}
\begin{enumerate}
\item Then $\mathfrak s=\mathfrak a^{\uplus 1/2}$ is a semicircular element with variance $\eta(b)=aba$;
\item Assume that
$\eta\colon\mathcal B\to\mathcal B$ is given by $\eta(b)=\frac1m\sum_{j=1}^ma_jba_j$ for a selfadjoint $n$-tuple $(a_1,\dots,a_n)\in\mathcal B^n$. Then the semicircular operator valued random variable $\mathfrak s$ with variance $\eta$ satisfies $\mathfrak s={\rm tr_m}(\mathfrak a^{\uplus 1/2}),$ where $\mathfrak a$ is the centered arcsine distribution with values in $M_m(\mathcal B)$ having variance $b\mapsto{\rm diag}(a_1,\dots,a_m)\cdot b \cdot{\rm diag}(a_1,\dots,a_m)$, $b\in M_m(\mathcal B).$
\end{enumerate}
\end{prop}

Given an $M_m(\mathcal B)$-valued distribution $\mathfrak N$, we define $\mathfrak n={\rm tr_m}(\mathfrak N)$ as the distribution satisfying
$\mathfrak n(\X b_1\X b_2\cdots\X b_q\X)={\rm tr_m}\mathfrak N(\X(b_1\otimes{1_m})\X(b_2\otimes{1_m})\cdots\X(b_q\otimes{1_m})\X)$. This is a distribution over $\mathcal B$. We observe that for any $b\in\mathcal B$,
\begin{eqnarray*}
{\rm tr_m}\mathfrak H_\mathfrak N(b\otimes{1_m}) & = & {\rm tr_m}\sum_{n=0}^\infty\mathfrak N((b\otimes{1_m})[\X(b\otimes{1_m})]^n)\\
& = & \sum_{n=0}^\infty\mathfrak n((b\otimes{1_m})[\X(b\otimes{1_m})]^n)\\
& = & \mathfrak H_\mathfrak n(b).
\end{eqnarray*}
Clearly, same result will hold for the generalized Cauchy transforms of $\mathfrak N$ and $\mathfrak n$. We note that the trace of a semicircular distribution with variance $\eta((b_{ij}))={\rm diag}(a_1,\dots,a_m)\cdot(b_{ij})\cdot{\rm diag}(a_1,\dots,a_m)$ is still semicircular: if we look at the characterization from Corollary \ref{comb-semicirc}, it follows immediately from the nature of the recurrences (1s)--(3s) that all elements $\mathfrak S(\X(b\otimes{1_m})\X(b\otimes{1_m})\cdots\X(b\otimes{1_m})\X$ will be diagonal matrices in $M_m(\mathcal B)$. Thus, taking ${\rm tr_m}$ in \eqref{S} from Corollary \ref{cor2.4}, $b\otimes{1_m}=F_\mathfrak S(b\otimes{1_m})+\eta(G_\mathfrak S(b\otimes{1_m}))$, will provide us with an equation
$$
b=\frac1m\sum_{j=1}^m(F_\mathfrak S(b\otimes{1_m}))_{jj}+
\frac1m\sum_{j=1}^ma_j(G_\mathfrak S(b\otimes{1_m}))_{jj}a_j.
$$
We conclude that $b\mapsto \frac1m\sum_{j=1}^m(G_\mathfrak S(b\otimes{1_m}))_{jj}$ is the Cauchy transform of the semicircular distribution $\mathfrak s={\rm tr_m}(\mathfrak S)$. This argument can be applied to the arcsine distribution as well, according to Theorem \ref{ga}.

\begin{proof}
The proof of (1) is straightforward.
Let $\omega(b)=
\frac12(b+F_{\B\boxplus\B}(b))$, $b\in\mathbb H^+(\cB)$. Expanding
by using the definition
of $F_\B$ and \eqref{o},
\eqref{F} gives:
\begin{eqnarray*}
F_{\B\boxplus\B}(b) & = & F_\B\left(\frac12(b+F_{\B\boxplus\B}(b))
\right)\\
& = & \left[(b+F_{\B\boxplus\B}(b)-2a)^{-1}+(b+F_{\B\boxplus\B}(b)+
2a)^{-1}\right]^{-1}\\
& = & \frac12(b+F_{\B\boxplus\B}(b)+2a)(b+F_{\B\boxplus\B}(b))^{-1}
(b+F_{\B\boxplus\B}(b)-2a)\\
& = & \frac12(b+F_{\B\boxplus\B}(b))+a-a-2a(b+F_{\B\boxplus\B}(b))^{-1}
a\\
& = & \frac12(b+F_{\B\boxplus\B}(b))-2a(b+F_{\B\boxplus\B}(b))^{-1}a.
\end{eqnarray*}
Now replacing $\omega$ in the above yields
$$
2\omega(b)-b = \omega(b)-a\omega(b)^{-1}a,
$$
or, equivalently
\begin{equation}
\omega(b)+a\omega(b)^{-1}a=b,\quad b\in\mathbb H^+(\cB).
\end{equation}
This is exactly equation \eqref{S} providing, according to
Corollary \ref{cor2.4}, the reciprocal of the operator-valued
Cauchy transform of the centered semicircle with variance
$\eta(b)=aba$. This, together with Proposition \ref{5} and Theorem \ref{4}, proves part (1).

To prove part (2), let us define the distribution
$\B$ on $M_m(\mathcal B)$ simply by
$$
G_\B(b)=\frac12\left[(b-{\rm diag}(a_1,\dots,a_m))^{-1}+(b+{\rm diag}(a_1,\dots,a_m))^{-1}\right],
$$
for $b\in\mathbb H^+(M_n(\mathcal B))$. Thus, now we view our scalar algebra to be directly $M_m(\mathcal B)$ and build the fully matricial structure starting from $M_m(\mathcal A)$, $M_m(\mathcal B)$ and $E_{M_m(\mathcal B)}$. As shown in the proof of part (1), it follows that there exists an $M_m(\mathcal B)$-valued semicircular random variable $\mathfrak S$ which is centered and has variance $\eta_m\colon M_m(\mathcal B)\to M_m(\mathcal B)$ given by $\eta_m(b) ={\rm diag}(a_1,\dots,a_m)\cdot b\cdot{\rm diag}(a_1,\dots,a_m)$, $b\in M_m(\mathcal B)$. We shall define $\tilde{\mathfrak s}$ to simply be $\mathfrak S$ viewed as taking values in $\mathcal B$. This gives us a variance for $\tilde{\mathfrak s}$ equal to
\begin{eqnarray*}
\lefteqn{E_{\mathcal B\otimes1_m}(\tilde{\mathfrak s}b\otimes1_m\tilde{\mathfrak s})  = }\\
& &
\left(\begin{array}{cccc}
a_1 & 0 & \cdots & 0\\
0 & a_2 & \cdots & 0\\
\cdots & \cdots & \cdots & \cdots \\
0 & 0 & \cdots & a_m\\
\end{array}\right)
\left(\begin{array}{cccc}
b & 0 & \cdots & 0\\
0 & b & \cdots & 0\\
\cdots & \cdots & \cdots & \cdots \\
0 & 0 & \cdots & b\\
\end{array}\right)
\left(\begin{array}{cccc}
a_1 & 0 & \cdots & 0\\
0 & a_2 & \cdots & 0\\
\cdots & \cdots & \cdots & \cdots \\
0 & 0 & \cdots & a_m\\
\end{array}\right)=\\
&  &
\left(\begin{array}{cccc}
a_1ba_1 & 0 & \cdots & 0\\
0 & a_2ba_2 & \cdots & 0\\
\cdots & \cdots & \cdots & \cdots \\
0 & 0 & \cdots & a_mba_m\\
\end{array}\right).
\end{eqnarray*}
Thus, taking partial traces gives
$$
b=\textrm{tr}_m(b\otimes 1_m)=\textrm{tr}_m(F_{\tilde{\mathfrak s}}
(b\otimes 1_m))+\textrm{tr}_m(\eta_m(G_{\tilde{\mathfrak s}}
(b\otimes 1_m))=F_\mathfrak s(b)+\frac1m\sum_{j=1}^ma_jG_\mathfrak s
(b)a_j.
$$
This proves (2) and completes our proof.
\end{proof}

We note that whenever $\mathcal B$ is finite dimensional, the above proposition gives a {\em complete} characterization of the correspondence between operator-valued semicircle, arcsine and Bernoulli distributions. This follows directly from Stinespring's dilation theorem.

We note that the relation described in Theorem  \ref{4} cannot hold unless the variance is of the form $b\mapsto aba$. Indeed, generally, this equality implies (by an identification of moments of order 2 already) that $\frac{b\eta(b)b\eta(b)b}{2}=\frac{b\eta(b)b\eta(b)+b\eta(b\eta(b)b)b}{4}$. This requires that $\eta(b)b\eta(b)=\eta(b\eta(b)b)$, which (for example for a $\mathcal B$ which is a factor) holds only when $\eta(b)=aba$.

We conclude with a remark about the dynamical system properties of the reciprocal of the Cauchy transform of $\mathfrak a$.
\begin{remark}
\begin{enumerate}
\item $F_\mathfrak a$ embeds in a composition semigroup,
i.e. there exists $F(t,b)\colon[0,+\infty)\times\mathbb H^+(B)\to
\mathbb H^+(B)$ so that $F(1,b)=F_\mathfrak a(b)$, $F(0,b)=b$,
and $F(t+s,b)=F(t,F(s,b))$. In particular, $F_\mathfrak a^{\circ
n}(b)=F(n,b)$ for all $n\in\mathbb N$.
Indeed, this follows from Corollary \ref{monstable}: replacing $2$ in it with any other natural number we relate $F_\mathfrak a^{\circ m}(b)=\sqrt{m}F_\mathfrak a(b/\sqrt{m})$. This holds for $\mathfrak a$ of any variance, so we have shown that $F_\mathfrak a$ embeds in a semigroup $F(t,b)\colon\mathbb Q_+\times\mathbb H^+(\mathcal B)\to
\mathbb H^+(\mathcal B)$ so that $F(1,b)=F_\mathfrak a(b)$, $F(0,b)=b$,
and $F(t+s,b)=F(t,F(s,b))$. The extension to $[0,+\infty)$ follows by continuity.
\item A consequence of the above item is that $F(t,b)=\sqrt{t}F_\mathfrak a(b/\sqrt{t})$. In particular, $F_\mathfrak a$ embeds in an analytic semigroup.

\item For any analytic composition semigroup $F$ over $\mathbb H^+(\mathcal B)$ and any $b\in\mathbb H^+(\mathcal B)$, the linear operator $\partial_bF(1,b)$ on $\mathcal B$ is injective. Indeed, assume $c\in B\setminus\{0\}$ is so that $\partial_bF(1,b)c=0$.
Then
$$
\partial_bF(1+t,b)c=\partial_bF(t,F(1,b))\partial_bF(1,b)c=0,
$$
so $t\mapsto\partial_bF(1+t,b)c$ is constant on $[0,+\infty)$.
Since this function extends analytically to $(-1,+\infty)$, it must be constant on all this interval. But as $(t,b)\mapsto F(t,b)$ is analytic, it follows that for $t>0$ sufficiently small, $\partial_bF(t,b)$ is close in norm to the identity, hence bijective. So for $t>-1$ close to $-1$, $\partial_bF(1+t,b)$ is close to the identity on $\mathcal B$. This contradicts $c\neq0$.
\item In particular, $\partial_tF(t,b)$ can never be zero.
Otherwise,
$$
\partial_tF(t+s,b)=\partial_t(F(s,F(t,b)))=\partial_bF(s,F(t,b))
\partial_tF(t,b)=0
$$
for all $s\in(-t,+\infty)$, so that $t\mapsto F(t,b)$ is constant.
Contradiction.\footnote{In fact, this IS possible, namely when
$b$ is a fixed point for some, hence all, $F(t,\cdot)$. However,
this would not allow $\Im F_\mathfrak a(b)>\Im b$
for all $b\in\mathbb H^+(\mathcal B)$, a relation satisfied for any invertible variance.}

\item $F_\mathfrak a$ is injective on all of $\mathbb H^+(B)$.
Indeed, assume that $F_\mathfrak a(b)=F_\mathfrak a(c).$
Then $F(t-1,F_\mathfrak a(b))=F(t,b)=\sqrt{t}F_\mathfrak a(b/\sqrt{t})
$ implies that
$$\sqrt{t}F_\mathfrak a(b/\sqrt{t})
=F(t-1,F_\mathfrak a(b))=F(t-1,F_\mathfrak a(c))
=\sqrt{t}F_\mathfrak a(c/\sqrt{t})
$$
for all $t\ge 1$, and hence, by analyticity, for all $t>0$.
Letting $t$ tend to zero, we obtain $b=c$, as claimed.
(Observe that in fact this holds true for all analytic
semigroups.)

\item If $\mathfrak a$ has variance $E_\mathcal B(\mathfrak a b\mathfrak a)=\eta(b)$, then $F_\mathfrak a$ satisfies
$$
F_\mathfrak a(b)=F_\mathfrak a'(b)(b-2\eta(b^{-1})),\quad b\in\mathbb
H^+(\mathcal B).
$$
If we let $M_\mathfrak a(b)=G_\mathfrak a(b^{-1})$, then the equation
above becomes
$$
M_\mathfrak a'(b)(b-2b\eta(b)b)=M_\mathfrak a(b).
$$
\item Finally, we mention the pde satisfied by $F_\mathfrak a$:
$$
\partial_tF(t,b)=\frac{1}{2t}(F(t,b)-\partial_bF(t,b)b),\quad
t\ge0,b\in\mathbb H^+(\mathcal B).
$$

\end{enumerate}
\end{remark}

\end{document}